\documentclass[12pt]{article}

\usepackage[cp1251]{inputenc}   
\usepackage[english]{babel}
\usepackage{amsmath,amsfonts,amssymb, amsthm}

\usepackage{cite}

\theoremstyle{plain}
\newtheorem{theorem}{Theorem}[section]
\newtheorem{lemma}{Lemma}[section]
\newtheorem{corollary}{Corrollary}[theorem]
\theoremstyle{definition}

\newtheorem{example}{Example}[section]

\begin{document}

\author{I.\,L.\,Sokhor}

\title{The continuation of finite $E_\mathfrak{F}$-groups theory}

\date{}

\maketitle

{\footnotesize

\textbf{Abstract.} 
We describe the structure of finite groups with
$\mathfrak{F}$\nobreakdash-\hspace{0pt}subnormal or
self-normalizing primary cyclic subgroups
when $\mathfrak{F}$ is a subgroup-closed saturate superradical formation
containing all nilpotent groups.
We prove that groups with absolutely
$\mathfrak{F}$\nobreakdash-\hspace{0pt}subnormal or
self-normalizing primary cyclic subgroups are soluble
when $\mathfrak{F}$ is a subgroup-closed saturate formation
containing all nilpotent groups.



}

\section{Introduction}

All groups in this paper are finite.
We use the standard notations and terminology of~\cite{Hup}.

Let $\mathfrak F$ be a formation and let $G$ be a group.

A subgroup~$H$ of $G$ is $\mathfrak F$\nobreakdash-\hspace{0pt}subnormal in $G$
if $G=H$ or there is a chain of subgroups
\[H=H_0\lessdot \  H_1\lessdot \  \ldots \lessdot  H_n=G    \]
such that $H_i/({H_{i-1}})_{H_i}\in \mathfrak F$ for all~$i$.
(or, equivalently, $H_i^\mathfrak F\le H_{i-1}$).
Here we write $A\lessdot B$ if $A$ is a maximal subgroup of a group~$B$,
and $A_B=\bigcap _{b\in B}A^b$ is the core of $A$ in~$B$,

A subgroup~$H$ of~$G$ is $\mathfrak F$\nobreakdash-\hspace{0pt}abnormal in $G$
if $L/K_L\not \in \mathfrak F$ for all $K$ and $L$ such that $H\le K\lessdot L\le G$.

A group $G$ is said to be an $E_\mathfrak{F}$-group if $G\notin\mathfrak{F}$
and its every non-trivial subgroup is $\mathfrak{F}$-subnormal or
$\mathfrak{F}$-abnormal in $G$.
The structure of $E_\mathfrak{F}$-groups for various formations $\mathfrak{F}$
is investigated by many authors, see the review paper of A.\,N.\,Skiba~\cite{Skiba2016}.

It is evident that in a group any proper subgroup
could not be $\mathfrak{F}$-subnormal and $\mathfrak{F}$-abnormal
at the same time, i.\,e. this notations are alternative.
If $\mathfrak{F}$ is a subgroup-closed formation containing
all nilpotent subgroups, then every $\mathfrak{F}$\nobreakdash-\hspace{0pt}abnormal
subgroup is self-normalizing,
i.\,e. it coincides with own normalizer.
But self-normalizingness and
$\mathfrak{F}$-subnormality are not alternative notions.
For example, every non-normal subgroup of prime index
in a soluble group is self-normalizing and $\mathfrak{U}$-subnormal.
Here $\mathfrak{U}$ is the formation of all supersoluble groups.

Groups with $\mathfrak F$\nobreakdash-\hspace{0pt}subnormal
or self-normalizing subgroups were studied in~\cite{Mon2016,MonS2017}.
In particular, V.\,S.\,Monakhov~\cite{Mon2016} showed that
the class of groups with $\mathfrak{U}$-subnormal
or self-normalizing primary subgroups is much wider than
the class of $E_\mathfrak{U}$-groups.

In this paper, we continue the research in noted theme.
We describe the structure of groups with
$\mathfrak{F}$-subnormal or self-normalizing
primary cyclic subgroups
when $\mathfrak{F}$ is a subgroup-closed saturate superradical formation
containing all nilpotent groups.
In addition, we prove that groups with absolutely
$\mathfrak{F}$\nobreakdash-\hspace{0pt}subnormal or
self-normalizing primary cyclic subgroups are soluble
when $\mathfrak{F}$ is a subgroup-closed saturate formation
containing all nilpotent groups.

\section{Preliminaries}

Let $G$ be a group. We denote the set of all prime devisors of $|G|$ by $\pi(G)$,
$A\rtimes B$ denotes the semidirect product of a normal subgroup $A$ and a subgroup~$B$.

The formations of all abelian and nilpotent groups are denoted by
$\mathfrak{A}$ and $\mathfrak{N}$, respectively.

A normal subgroup-closed formation $\mathfrak{F}$ is called
superradical if  any group $G=AB$ belongs to $\mathfrak{F}$
whenever  $A$ and $B$ are $\mathfrak{F}$\nobreakdash-\hspace{0pt}subnormal
$\mathfrak{F}$\nobreakdash-\hspace{0pt}subgroups of $G$.
It is known that formations with Shemetkov property and lattice formations
are superradical.

Let $\mathfrak F$ be a formation and $G$ be a group.
The intersection of all normal subgroups of~$G$
with quotient in $\mathfrak F$ is called
$\mathfrak F$\nobreakdash-\hspace{0pt}residual and
is denoted by $G^\mathfrak F$.

We need the following properties of  $\mathfrak{F}$-subnormal
and $\mathfrak{F}$-abnormal subgroups.

\begin{lemma}\label{lem_Fsn}
Let $\mathfrak{F}$ be a formation,
let $H$ and $K$ be subgroups of a group~$G$,
and let $N$ be a normal subgroup of $G$.
The following statements hold.

$(1)$~If $K$ is $\mathfrak{F}$\nobreakdash-\hspace{0pt}subnormal
    in $H$ and $H$ is $\mathfrak{F}$\nobreakdash-\hspace{0pt}subnormal in~$G$,
    then $K$ is $\mathfrak{F}$\nobreakdash-\hspace{0pt}subnormal in~$G$
    \textup{\cite[6.1.6\,(1)]{BalCl}}.

$(2)$~If $K/N$ is $\mathfrak{F}$\nobreakdash-\hspace{0pt}subnormal in $G/N$,
    then  $K$ is $\mathfrak{F}$\nobreakdash-\hspace{0pt}subnormal in~$G$
    \textup{\cite[6.1.6\,(2)]{BalCl}}.

$(3)$~If $H$ is $\mathfrak{F}$\nobreakdash-\hspace{0pt}subnormal in~$G$,
    then $HN/N$ is $\mathfrak{F}$\nobreakdash-\hspace{0pt}subnormal in~$G/N$
    \textup{\cite[6.1.6\,(3)]{BalCl}}.

$(4)$~If $\mathfrak{F}$ is a subgroup-closed formation and $G^{\mathfrak{F}}\leq H$,
    then $H$ is $\mathfrak{F}$\nobreakdash-\hspace{0pt}subnormal in~$G$
    \textup{\cite[6.1.7\,(1)]{BalCl}}.

$(5)$~If $\mathfrak{F}$ is a subgroup-closed formation and
    $H$ is $\mathfrak{F}$\nobreakdash-\hspace{0pt}subnormal in~$G$,
    then $H\cap K$ is $\mathfrak{F}$\nobreakdash-\hspace{0pt}subnormal in~$K$,
    \textup{\cite[6.1.7\,(2)]{BalCl}};

$(6)$~If $\mathfrak{F}$ is a subgroup-closed formation, $K\leq H$,
    $H$ is $\mathfrak{F}$\nobreakdash-\hspace{0pt}subnormal in~$G$
    and $H\in\mathfrak{F}$, then $K$ is
    $\mathfrak{F}$\nobreakdash-\hspace{0pt}subnormal in~$G$.
\end{lemma}

\begin{lemma}[{\cite[Lemma 1.4]{Mon2016}}]\label{lem_Fabn}
Let $\mathfrak F$ be a subgroup-closed formation containing groups
of order~$p$ for all~$p\in \mathbb{P}$, and let $A$ be an
$\mathfrak F$\nobreakdash-\hspace{0pt}abnormal subgroup of a group~$G$.
The following statements hold.

    $(1)$ If $A\le B\le G$, then $B$ is
    $\mathfrak F$\nobreakdash-\hspace{0pt}abnormal in $G$ and $B=N_G(B)$;

    $(2)$ If $G$ is soluble, then $A$ is abnormal in $G$.
\end{lemma}

A subgroup $H$ of a group $G$ is abnormal
if $x \in \langle  H,H^x \rangle$for any $x \in G$.
An abnormal subgroup is self-normalizing.

\begin{lemma}\label{lem_abn}
Let $G$ be a group. The following statements hold.

$(1)$ If $P$ is a Sylow subgroup of $G$,
then $N_G(P)$ is abnormal in $G$.

$(2)$ If $A$ is an abnormal subgroup of $G$ and $A\le  B\le G$,
then $B$ is abnormal in $G$ and $N_G(B)=B$.

$(3)$ If $A$ is an abnormal subgroup of $G$
and $N$ is a normal subgroup in $G$,
then $AN/N$ is abnormal in $G/N$.
\end{lemma}

A Carter subgroup is a nilpotent self-normalizing subgroup~\cite[VI.12]{Hup}.
In soluble groups, Carter subgroups exist and are conjugate.
An insoluble group can have no Carter subgroups,
but if they exist, then they are conjugate~\cite{Vd}.

A group $G$ is a minimal non-$\mathfrak{F}$-group
if $G\notin\mathfrak{F}$ but every proper subgroup of $G$
belongs to $\mathfrak{F}$.
Minimal non-$\mathfrak{N}$\nobreakdash-\hspace{0pt}groups
are also called Schmidt groups, their property are well known~\cite{Mon_Sch}.

\begin{lemma}\label{lem_Fsn_max}
Let $\mathfrak{F}$ be a subgroup-closed saturate formation.
If every maximal subgroup of a group $G$ is $\mathfrak{F}$-subnormal,
then $G\in\mathfrak{F}$.
\end{lemma}

\begin{proof}
If $M$ is a maximal subgroup of $G$,
then $G/M_G\in\mathfrak{F}$. Hence $G/\bigcap M_G=G/\Phi(G)\in\mathfrak{F}$
and $G\in\mathfrak{F}$.
\end{proof}

\section{Groups with $\mathfrak{F}$-subnormal or self-normalizing primary subgroups}

\begin{lemma}\label{lem_inF}
Let  $\mathfrak{F}$ be a subgroup-closed saturate superradical formation
containing all nilpotent groups. A soluble group $G$ belongs to
$\mathfrak{F}$ if and only if every primary cyclic subgroup of $G$ is
$\mathfrak{F}$\nobreakdash-\hspace{0pt}subnormal.
\end{lemma}

\begin{proof}
If $G\in\mathfrak{F}$, then every proper subgroup
(including every primary cyclic subgroup) is $\mathfrak{F}$-subnormal in $G$.

Now, suppose that there are groups  that do not belong to $\mathfrak{F}$
but all their primary cyclic subgroup are $\mathfrak{F}$-subnormal.
Choice a group  $G$ of least order among them.
Hence every proper subgroup of $G$ belongs to $\mathfrak{F}$.
According to \cite[Lemma~3]{Sem96}, $G$ is a Schmidt group,
and  $G=P\rtimes \langle y\rangle$~\cite[Theorem~1.1]{Mon_Sch}.
In view of~\cite[Theorem~1.5]{Mon_Sch},
either $G^\mathfrak{F}\leq\Phi (G)$ or $P\leq G^\mathfrak{F}$.
If $G^\mathfrak{F}\leq\Phi (G)$,
then $G\in\mathfrak{F}$ as $\mathfrak{F}$ is a saturate formation,
a contradiction. Assume that $P\leq G^\mathfrak{F}$.
By the choice of $G$, $\langle y\rangle$ is
$\mathfrak{F}$\nobreakdash-\hspace{0pt}subnormal in $G$.
Therefore there is a maximal subgroup $M$ in $G$
that contains  $\langle y\rangle$ and $G^{\mathfrak{F}}$,
a contradiction. 
\end{proof}

\begin{theorem}\label{th_1}
If $\mathfrak{F}$ is a subgroup-closed saturate superradical formation
containing all nilpotent groups and $G\notin\mathfrak{F}$ is a soluble group,
then the following statements are equivalent.

$(1)$~Every primary cyclic subgroup of $G$ is
$\mathfrak{F}$\nobreakdash-\hspace{0pt}subnormal or self-normalizing.

$(2)$~Every non-abnormal  subgroup of $G$ is
$\mathfrak{F}$\nobreakdash-\hspace{0pt}subnormal and belongs to $\mathfrak{F}$.

$(3)$~$G=G^{\prime}\rtimes \langle x\rangle$, where
$\langle x\rangle$ is a Sylow $p$\nobreakdash-\hspace{0pt}subgroup
for a prime $p\in\pi(G)$ and a Carter subgroup, $G'=G^\mathfrak{N}$,
$G'\rtimes \langle x^p\rangle\in \mathfrak{F}$.
\end{theorem}

\begin{proof}
$(1)\Rightarrow (3)$:
Assume that every primary cyclic subgroup of a soluble group $G\notin\mathfrak{F}$
is $\mathfrak{F}$\nobreakdash-\hspace{0pt}subnormal or self-normalizing.
By Lemma~\ref{lem_inF}, for some $p\in \pi (G)$ there is
a cyclic $p$\nobreakdash-\hspace{0pt}subgroup $\langle x\rangle$
that is not $\mathfrak{F}$\nobreakdash-\hspace{0pt}subnormal in $G$.
By the choice of $G$,  $\langle x\rangle$ is self-normalizing.
It implies that $\langle x\rangle$ is a Sylow subgroup and
a Carter subgroup of $G$.
In view of~\cite[IV.2.6]{Hup}, there is a normal Hall
$p^{\prime}$\nobreakdash-\hspace{0pt}subgroup $G_{p^{\prime}}$ of $G$, and
$G=G_{p^{\prime}}\rtimes\langle x\rangle$.
Clearly, $G^{\mathfrak N}\le G^\prime \le G_{p^\prime}$.
Since $G/G^{\mathfrak N}$ is nilpotent and
$PG^{\mathfrak N}/G^{\mathfrak N}$ is
a Sylow subgroup of $G/G^{\mathfrak N}$,
we conclude that $PG^{\mathfrak N}$ is normal in~$G$.
In view of the Frattini Lemma,
\[G=N_G(P)(PG^{\mathfrak N})=PG^{\mathfrak N}=G_{p^\prime }\rtimes P.\]
So $G^{\mathfrak N}= G^\prime = G_{p^\prime }$ and
$G=G'\rtimes \langle x\rangle$.

By~\cite[VI.12.2]{Hup},
$G^{\prime}\rtimes \langle x^p\rangle$
has no self-normalizing primary cyclic subgroups.
Hence from the choice of $G$
it follows that every primary cyclic subgroup of
$G^{\prime}\rtimes \langle x^p\rangle$
is $\mathfrak{F}$-subnormal in~$G$, and
$G^{\prime}\rtimes \langle x^p\rangle\in\mathfrak{F}$
according to Lemma~\ref{lem_inF}.

$(3)\Rightarrow (2)$:
Assume that a soluble group $G\notin\mathfrak{F}$ satisfies Statement $(3)$.
Let $H$ be a non-abnormal subgroup of $G$.
Since $G'$ is normal in $G$ and $G$ is soluble,
it implies that $G'$ is $\mathfrak{F}$-subnormal in $G$ by~\cite[Lemma~1.11]{Mon2016}.
If $H\leq G'$, then $H\in\mathfrak{F}$ and $H$ is $\mathfrak{F}$-subnormal in $G$
by Lemma~\ref{lem_Fsn}. Suppose that $H$ is not contained in $G'$.
If $G=G'H$, then $H$ contains a Carter subgroup of $G$
that is a Sylow subgroup of $G$. Therefore $H$ is abnormal
according to Lemma~\ref{lem_abn}. This contradicts the choice of $H$.
Hence $G'H$ is a proper subgroup of $G$. According to the choice of $G$,
$G'H\in\mathfrak{F}$ and $H$ is $\mathfrak{F}$-subnormal in $G$
by Lemma~\ref{lem_Fsn}.

$(2)\Rightarrow (1)$:
It is evident according to self-normalizing of abnormal subgroups.
\end{proof}

\begin{corollary}
Let $\mathfrak{F}$ be a subgroup-closed saturate superradical formation
containing all nilpotent groups.
Assume that every primary cyclic subgroup of
a soluble group $G\notin\mathfrak{F}$
is $\mathfrak{F}$-subnormal or self-normalizing,
$K$ is a Carter subgroup of~$G$
and $A$ is a proper  subgroup of~$G$.
The following statements hold.

$(1)$~If $|K|$ divides $|A|$, then $A$ is abnormal.

$(2)$~If $|K|$ does not divide $|A|$, then $A$ is $\mathfrak{F}$-subnormal
and $A\in\mathfrak{F}$.
\end{corollary}

Note that if $\mathfrak{F}$ is a subgroup-closed formation
containing all nilpotent subgroups, then in view of Lemma~\ref{lem_Fabn}
every $\mathfrak{F}$\nobreakdash-\hspace{0pt}abnormal subgroup is self-normalizing.
Hence Theorem~\ref{th_1} enables to describe the structure of
an $E_\mathfrak{F}$-group when $\mathfrak{F}$ is
a subgroup-closed saturate superradical formation
containing all nilpotent groups.

\begin{corollary}
If $\mathfrak{F}$ is a subgroup-closed saturate superradical formation
containing all nilpotent groups and $G\notin\mathfrak{F}$ is a soluble group,
then the following statements are equivalent.

$(1)$~Every primary cyclic subgroup of $G$ is
$\mathfrak{F}$\nobreakdash-\hspace{0pt}subnormal or $\mathfrak{F}$\nobreakdash-\hspace{0pt}abnormal.

$(2)$~$G$ is an $E_\mathfrak{F}$-group.

$(3)$~$G=G^{\prime}\rtimes \langle x\rangle$, where
$\langle x\rangle$ is a Sylow $p$\nobreakdash-\hspace{0pt}subgroup
for a prime $p\in\pi(G)$ and a Carter subgroup,
$G^{\prime}=G^{\mathfrak{F}}$ and
$G^{\prime}\rtimes \langle x^p\rangle\in \mathfrak{F}$.
\end{corollary}

\begin{example}
Let $\mathfrak{F}$ be the formation of all groups
with nilpotent derived subgroups.
We use $E_{p^n}$ to denote an elementary abelian group of order $p^n$
for a prime $p$ and a positive integer $n$,
$C_m$ is a cyclic group of order $m$ for a positive integer $m$.

Consider a group $G$~\cite[{SmallGroup~ID~(864,4670)}]{gap})
\[
G=(S_3\times S_3\times A_4)\rtimes C_2.
\]
In $G$, a Sylow $3$\nobreakdash-\hspace{0pt}subgroup
${G_3 \simeq E_{3^3}}$ is $\mathfrak{F}$\nobreakdash-\hspace{0pt}subnormal,
a Sylow $2$\nobreakdash-\hspace{0pt}subgroup
$G_2 \simeq E_{2^4}\rtimes C_2$ is self-normalizing, but
$G_2$ is not $\mathfrak{F}$\nobreakdash-\hspace{0pt}subnormal and
is not $\mathfrak{F}$\nobreakdash-\hspace{0pt}abnormal.
Every proper subgroup of $G_2$ is
$\mathfrak{F}$\nobreakdash-\hspace{0pt}subnormal in $G$.
In addition,
\[G^{\mathfrak{F}}=F(G)\simeq E_{3^2}\times E_{2^2}<
G^{\mathfrak{N}}\simeq E_{3^2}\times A_4<
G^{\prime}\simeq (E_{3^2}\times A_4)\rtimes C_2.
\]
Thus, $G$ belongs to the class of groups with
$\mathfrak{F}$\nobreakdash-\hspace{0pt}subnormal or
self-normalizing primary subgroups,
and $G$ does not belong to the class of groups with
$\mathfrak{F}$\nobreakdash-\hspace{0pt}subnormal or
$\mathfrak{F}$\nobreakdash-\hspace{0pt}abnormal primary subgroups.
\end{example}


\section{Groups with absolutely $\mathfrak{F}$-subnormal
or self-normalizing primary subgroups}

A.\,F.\,Vasil'ev proper~\cite{VasPFMT2019} the following concept.

Let $\mathfrak{F}$ be a formation.
A subgroup $H$ of a group $G$ is called absolutely $\mathfrak{F}$-subnormal
in $G$ if any subgroup $L$ containing $H$ is $\mathfrak{F}$-subnormal
in $G$.

In view of {\cite[Corrolary~3.2]{VasPFMT2019}},
the following lemma hold.

\begin{lemma}\label{lem_absF_Cycl}
Let $\mathfrak{F}$ be a subgroup-closed saturate formation
containing all nilpotent groups.
A group $G$ belongs to $\mathfrak{F}$ if and only if
every primary cyclic subgroup of $G$ is
absolutely $\mathfrak{F}$-subnormal in $G$.
\end{lemma}

\begin{theorem}\label{th_2}
Let $\mathfrak{F}$ be a subgroup-closed saturate formation
containing all nilpotent groups.
Every primary cyclic subgroup of a group $G\notin\mathfrak{F}$
is absolutely $\mathfrak{F}$-subnormal or self-normalizing
if and only if $G$ is a non-nilpotent group
all proper subgroups of which is primary,
in particular, $G=G'\rtimes \langle x\rangle$,
$G'$ is an elementary abelian $p$-group for a prime $p\in\pi(G)$,
$\langle x\rangle$ is a maximal subgroup of order $q$
and a Carter subgroup of $G$ for a prime $q\in\pi(G)$, $q\neq p$.
\end{theorem}

\begin{proof} 
Assume that every primary cyclic subgroup
of a group $G\notin\mathfrak{F}$ is absolutely $\mathfrak{F}$-subnormal
or self-normalizing. Since $\mathfrak{N}\subseteq\mathfrak{F}$,
clearly $G\notin\mathfrak{N}$.
If every primary cyclic subgroup of $G$ is
absolutely $\mathfrak{F}$-subnormal in $G$,
then in view of Lemma~\ref{lem_absF_Cycl}
$G\in\mathfrak{F}$, a contradiction.
Consequently, for a prime $q\in\pi(G)$
there is a cyclic $q$-subgroup $Q=\langle x\rangle$
that is not absolutely $\mathfrak{F}$-subnormal in $G$.
By the choice of $G$, $Q$ is self-normalizing,
and so $Q$ is a Sylow subgroup and a Carter subgroup of $G$.
According to~\cite[IV.2.6]{Hup}, there is a Hall $q'$-subgroup $G_{q'}$
such that $G=G_{q'}\rtimes Q$. Clearly, $G^\mathfrak{F}\leq G'\leq G$.
Since $G/G^{\mathfrak N}$ is nilpotent and
$QG^{\mathfrak N}/G^{\mathfrak N}$ is a Sylow subgroup
of $G/G^{\mathfrak N}$, we obtain that $QG^{\mathfrak N}$ is normal in~$G$.
By the Frattini Lemma,
\[
G=N_G(Q)(QG^{\mathfrak N})=QG^{\mathfrak N}=G_{q\prime }\rtimes Q.
\]
Consequently, $G^{\mathfrak N}= G_{p\prime}=G'$.

Let $A$ be a maximal subgroup in $Q$.
By the choice of $G$, $A$ is absolutely $\mathfrak{F}$-subnormal
or self-normalizing in $G$.
If $A$ is self-normalizing, then $A$ is a Carter subgroup,
and $A$ is conjugate with $Q$~\cite{Vd}, a contradiction.
Hence $A$ is absolutely $\mathfrak{F}$-subnormal in $G$, and
$Q$ is absolutely $\mathfrak{F}$-subnormal in $G$,
a contradiction. Therefore $|Q|=q$ and $G$ is soluble in view of~\cite[IV.2.8]{Hup}.

Suppose that $Q$ is not a maximal subgroup of $G$.
Then every maximal subgroup $M$ of $G$
contains a primary cyclic $r$-subgroup $R$, $r\neq q$.
If $R$ is self-normalizing, then $R$ is a Carter subgroup of $G$
and $R$ is conjugate with $Q$ \cite[VI.12.2]{Hup},
a contradiction with $r\neq q$. Therefore $R$ is
absolutely $\mathfrak{F}$\nobreakdash-\hspace{0pt}subnormal in $G$,
and $M$ is $\mathfrak{F}$-subnormal in $G$.
Thus every maximal subgroup of $G$ is $\mathfrak{F}$-subnormal,
and $G\in\mathfrak{F}$ by Lemma~\ref{lem_Fsn_max}, a contradiction.
It implies $Q$ is a maximal subgroup,
and $G'$ is a minimal normal subgroup of $G$.
Since $G$ is soluble, we have $G'$ is an elementary abelian $p$-group
for a prime $p$, $p\neq q$.

Conversely, assume that $G=G'\rtimes \langle x\rangle$,
$G'$ is an elementary abelian $p$-group for a prime $p\in\pi(G)$,
$\langle x\rangle$ is a maximal subgroup of order $q$
and a Carter subgroup of $G$ for a prime $q\in\pi(G)$, $q\neq p$.
Let $A$ be a primary cyclic subgroup of $G$.
If $\pi(A)=q$, then $A$ is a Carter subgroup and~$A=N_G(A)$.
If $\pi(A)=p$, then $A\le G^\prime\in\mathfrak{A}$.
Suppose that $H$ is a proper subgroup of $G$ such that $A\leq H$.
Hence $H$ is subnormal in $G$, and $H$ is $\mathfrak{F}$-subnormal in $G$
in view of ~\cite[Lemma~1.11]{Mon2016}.
So $A$ is absolutely $\mathfrak{F}$-subnormal in $G$.
\end{proof}

\end{document}